\documentclass[letterpaper,journal,final]{IEEEtran}  

\usepackage{tabularx}
\usepackage[ruled,linesnumbered]{algorithm2e}
\usepackage{comment}
\usepackage{derivative}
\usepackage{stackengine}
\usepackage{amsmath}
\usepackage{mathtools}
\usepackage{autobreak}
\usepackage{chemformula}
\usepackage{amssymb}
\usepackage{graphicx}
\usepackage[english]{babel}
\usepackage[numbers,sort&compress]{natbib}

\newtheorem{thm}{Theorem}

\newtheorem{definition}{Definition}

\newtheorem{cor}{Corollary}

\newtheorem{assum}{Assumption}
\newtheorem{prop}{Proposition}
\newtheorem{fact}{Fact}

\DeclareMathOperator{\col}{col}

\DeclareMathOperator{\diag}{diag}
\DeclareMathOperator{\tr}{tr}
\newcommand{\R}{\ensuremath{\mathbb R}}

\newcommand{\eucnorm}[1]{\ensuremath{\vert #1 \vert}}
\newcommand{\specnorm}[1]{\ensuremath{\Vert #1 \Vert}}

\def\qedp{\hspace*{\fill}~{\tiny $\blacksquare$}} 
\def\qeds{\hspace*{\fill}~{\tiny $\Box$}}

\title{Semi-definite programs for online control of nonlinear systems with stability guarantees} 

\author{Xiaoyan Dai  
\thanks{Xiaoyan Dai is with the Department of Electrical Engineering, Tsinghua University, 100084 Beijing, China. Email: {\tt\small daixiaoyan@tsinghua.edu.cn.} 
}}

\begin{document}
\maketitle

\begin{abstract}
This paper develops a semidefinite-programming-based method for online feedback control of nonlinear systems using a state-dependent representation. We formulate sequences of time-varying SDPs whose optimal solutions jointly yield a stabilizing feedback controller and a Lyapunov certificate satisfying stability conditions and quadratic performance specifications. We further establish compact conditions certifying recursive feasibility of the resulting SDP sequences and derive estimates of the region of attraction. Numerical examples on representative nonlinear systems illustrate the flexibility and effectiveness of the proposed method. 
\end{abstract}


\section{Introduction}

Computational tools have been extensively used for control design, among which we focus on optimization-based methods. These methods formulate control design problems as optimization programs whose decision variables are control inputs or feedback control gains that map system measurements to control inputs. 
The optimization programs incorporate the plant dynamics 
and performance specifications as (dynamic) constraints and objective functions, respectively. 

The optimization-based control design has been prevalent, with problems ranging from optimal \cite{numericalh2}, predictive \cite{rawlings2002mpctutorial}, safe \cite{wabersich2018safelinear}, and distributed control \cite{dorfler2014sparsity} to data-driven control \cite{deformulas}. 
The corresponding formulations are also varied, ranging from linear matrix inequalities (LMIs) \cite{ScherLMIoutputfeedback} and sum of squares \cite{prajna2004nonlinear} to the mixed integer quadratic programming \cite{bemporad1999hybridcontrol}. 
Among these, LMI-based formulations are particularly attractive for their computational efficiency and flexibility. 
We refer the reader to \cite{boyd1994lmibook} and \cite{scherer2021linear} for LMI-based formulations of diverse control design problems for linear systems. 
The control-oriented LMIs typically incorporate stability guarantees and performance specifications for the \emph{closed-loop} system behavior. As a result, the solutions provide \emph{certifiable} controllers.

This motivates us to propose novel LMI-based formulations for nonlinear control systems, a much more demanding context.  
Specifically, we focus on the problem of stabilizing input-affine nonlinear systems 
while minimizing a quadratic cost on state deviation and control effort.  
The goal is to develop LMIs that shape the closed-loop behaviors to be certifiably stable under performance considerations. The solutions result in provably correct control that stabilizes input-affine nonlinear systems around the desired equilibrium with relatively less control energy. 

\emph{Related work.} 
Nonlinear matrix inequalities (NLMIs) have been proposed for $\mathcal{H}_\infty$ control of 
a class of nonlinear systems in a state dependent representation \cite{lu1995nlmi}. 
Note that the NLMIs therein are in fact \emph{state dependent} matrix inequalities. 
Meanwhile, the state dependent representation has been widely used in nonlinear optimal control, where algebraic Riccati equations are solved along the trajectory, as a practical alternative for solving Hamiltonian-Jacobi-Bellman equations, see \cite{pearson1962approximation, beeler2004nasa, cloutier1996sdretheo, dutka2005optimized, sassanoTAC2012, kaiser2021sdrekoopman}. 
We further mention that based on the state dependent representation,  \cite{tahirovic2022linearlike} performs nonlinear optimal control via policy iteration. 
More recently, using the state dependent representation, \cite{dai2021sdre} considers a data-driven online control setting where a data-dependent semidefinite program is solved along the trajectory and, if feasible, renders the origin of the closed-loop system asymptotically stable. 

Despite the progress, see the applications of state dependent Riccati equation (SDRE)-based control across industrial practices \cite{ccimen2008sdresurvey}, a major challenge encountered in this line of inquiry is \emph{certifying} the stability of the system in closed-loop. 
In addition, recursive feasibility is vital as it determines whether the control scheme can generate a feedback gain at each time step. Nevertheless, the investigation in this respect appears to be sparse.

\emph{Contribution.} 
This paper makes two main contributions. First, we develop semidefinite-programming-based methods for online feedback control of nonlinear systems using state-dependent representations. In particular, we formulate sequences of time-varying SDPs whose optimal solutions jointly yield feedback controllers and Lyapunov certificates satisfying stability conditions and quadratic performance specifications. Second, we establish compact conditions certifying recursive feasibility of the resulting SDP sequences and derive computable estimates of the region of attraction, which characterize admissible initialization sets for the closed-loop guarantees. 

More specifically, we first construct a sequence of state-dependent SDPs for scalar input-affine nonlinear systems whose solutions yield a time-varying controller. We show that the closed-loop system is geometrically stable when initialized within a derived RoA estimate. We then develop a related regularized SDP framework for higher-dimensional systems, whose solutions simultaneously provide stabilizing controllers and Lyapunov certificates. Finally, we derive compact conditions certifying recursive feasibility of the proposed SDP sequences and demonstrate the effectiveness of the methods on representative nonlinear systems, including control of a single-machine-infinite-bus power system, stabilization of a jet engine compressor, and an inverted pendulum with a state-dependent input vector field.

{\it{Outline.}} The structure of the paper is specified as follows. Section \ref{sec-pre-pro-form} provides preliminaries on the SDP formulation of the LQR problem and introduces the problem under consideration. Then we consider the case of scalars in Section \ref{scalar-sec}. In Section \ref{higher-dimen-sec}, we investigate the higher-dimensional case. 
In Section \ref{case-sec}, we demonstrate the results on three example systems. Finally, we summarize the paper and draw the conclusions in Section \ref{conclude-sec}.

\section{Preliminaries and problem statement}\label{sec-pre-pro-form}

\subsection{Notation}
We denote the set of real numbers and natural numbers including zero by $\R$ and $\mathbb N_0$, respectively. 
We use $\succ (\succeq)$ and $\prec (\preceq)$ to denote the positive and negative (semi)definiteness of the involved matrices, respectively. 
The $2$-norm of a vector $x$ is denoted by $\eucnorm{x}$ and the induced $2$-norm of a matrix $A$ is denoted by $\specnorm{A}$. The notation $\lambda_m(Q)$ and $\lambda_M(Q)$ are the minimum and maximum eigenvalues of a square matrix $Q$, respectively. We denote the set of $n \times n$ real-valued symmetric matrices by $\mathbb S^{n \times n}$. We use $M^\top$ to denote the transpose of $M$. For matrices $A=A^\top$, $B$, $C=C^\top$, we abbreviate the symmetric matrix 
$\left[\begin{smallmatrix} 
 A&B \\ B^\top & C \end{smallmatrix}\right]$ as $ 
\left[\begin{smallmatrix} 
A & B \\ * & C 
\end{smallmatrix} \right]
$. For matrices $M$, $N$, and $O$ of compatible dimensions, we abbreviate $MN O(MN)^\top$ to $MN \cdot O(\star)^\top$, where $\star$ clarifies unambiguously that $MN$ is the term to be transformed.

\subsection{SDP formulation for discrete-time LQR problem}\label{Sec:pre:SDP} 
 
Consider a discrete-time linear time invariant system 
\begin{align*}
    x_{t+1} = A x_t + B u_t, \quad x(0) = x_0, 
\end{align*}
with $x \in \R^n$ the state, $u \in \R^m$ the control input, $x_0 \in \R^n$ the initial condition, $t \in \mathbb N_0$; $(A,B)$ is stabilizable. 
The linear quadratic regulator (LQR) problem is to design the controller 
\begin{align*}
    u = K x 
\end{align*}
such that the origin of the closed-loop system 
\begin{align*}
    x_{t+1} = (A+BK)x_t, \quad x(0) = x_0
\end{align*}
is asymptotically stable while minimizing the cost function 
\begin{align} \label{cost-def-lti-lq}
    J(x_0, u) \coloneqq \sum_{t = 0}^{\infty} x_t^\top Q x_t + 
    u_t^\top R u_t 
\end{align}
with weighting matrices $Q \succ 0$ and $R \succ 0$. 
It is well known that the LQR control gain $K$ is {\em{unique}} \cite[Sec. 6.4]{chen1995optimal} and can be computed as 
\begin{align}\label{lti-lq}
     K := -(R + B^\top P_r B)^{-1} B^\top P_r A 
\end{align}
where $P_r$ is the unique positive definite solution to the discrete-time Riccati equation \cite[Thm. 6.3.2]{chen1995optimal}
\begin{align}\label{riccati-lti-lq}
{\small    A^\top P_r A - P_r 
  - A^\top P_r B
  (R + B^\top P_r B)^{-1} B^\top P_r A + Q = 0 . }
\end{align}

On the other hand, finding the LQR controller $K$ can also be viewed as an $\mathcal{H}_2$ problem. 
To see this, we consider the performance output $y := \col(Q^{1/2} x, R^{1/2} u)$  
and the system 
\begin{align}\left[ \begin{array}{c}x_{t+1} \\ y_t \end{array} \right] = \left[ \begin{array}{c|c} A+BK & I \\ \hline \left[ \begin{array}{c} Q^{1/2} \\ R^{1/2} K \end{array} \right] & 0 \end{array} \right] \left[ \begin{array}{c} x_t \\ d_t \end{array} \right], \label{h2phisys} \end{align}
where $d = \delta_d x_0$ with 
$\delta_d$ the discrete-time unit pulse applied at time $t = 0$ \cite[Sec. 6.4]{chen1995optimal}. 
An SDP has been built \cite[eq. 9]{deformulas}
\begin{subequations}\label{lmi-lq}
\begin{align}
&\min\limits_{K, P, L, \gamma} \quad \gamma
\\&\text{s.t.} \notag \\
& (A+BK) P (A+BK)^\top - P + I \preceq 0 \label{lti-stab1}
\\ 
& P \succeq I \label{lti-stab2} \\ & L - K P K^\top \succeq 0 \label{lti-perf1} \\
& \gamma \geq \tr (Q P) + \tr (R L) \label{lti-perf2}
\end{align}\end{subequations}%
whose optimal solution $K$ minimizes the $\mathcal{H}_2$ norm of the system in \eqref{h2phisys}. Moreover, the optimal $K$ coincides with the LQR control gain. 
We summarize below the properties of the SDP-based formulation \eqref{lmi-lq} in relation to the Riccati-based formulation in \eqref{lti-lq} and \eqref{riccati-lti-lq}. 

\begin{prop}\label{riccati-sdp-lq-property}
    The SDP \eqref{lmi-lq} has a unique optimal solution $(K, P, L, \gamma)$ where $K$ coincides with the one defined in \eqref{lti-lq} and $\gamma$ equals to $\tr(P_r)$. \qeds
\end{prop}

\emph{Proof}. See Appendix \ref{appdx:proof:formulation-relation}. 
\qedp

\subsection{Problem formulation}

Consider the input-affine nonlinear system 
\begin{align}\label{n-sys-general-form}
    x_{t+1} =  f(x_t) + B(x_t)u_t, \quad x(0) = x_0,
\end{align}
with $x \in \R^n$ the state, $u \in \R^m$ the control input, $t \in \mathbb N_0$, $x_0 \in \R^n$ the initial condition, $f: \R^n \rightarrow \R^n$ a continuously differentiable function with $f(0)=0$ and $B: \R^n \rightarrow \R^{n \times m }$ a continuous function. 
The problem we are interested in is to design a state feedback controller in the form of
\begin{align}
    u_t = K(x_t)x_t \label{general-controller}
\end{align}
such that the origin of the closed-loop system 
\begin{align}\label{general-closed}
    x_{t+1} = f(x_t) + B(x_t)K(x_t) x_t 
\end{align}
is asymptotically stable. \qeds

Under the continuously differentiable condition on $f$, we can derive a linear-like representation for the system described in \eqref{n-sys-general-form}. 

\begin{fact}\label{form-lem}
    Consider the nonlinear system \eqref{n-sys-general-form}.
    There exists a continuous map $A: \R^n \rightarrow \R^{n\times n}$ such that 
    \begin{align}\label{f-A-x}
        f(x) = A(x)x. 
    \end{align}
\end{fact}

By Fact \ref{form-lem},  the system \eqref{n-sys-general-form} can be written as 
    \begin{align}\label{n-sys}
    x_{t+1} = A(x_t)x_t + B(x_t)u_t, 
\end{align}
with the continuous map $A(x)$ given in \eqref{f-A-x}.
Similarly, the closed-loop system \eqref{general-closed} reads as
\begin{align}\label{n-close-sys}
    x_{t+1} = \underbrace{(A(x_t) + B(x_t) K(x_t))}_{A_{cl}(x_t) :=} x_t. 
\end{align} 
The linear-like representation in \eqref{n-sys} and 
performance considerations suggest designing the controller in \eqref{general-controller} by iteratively solving the LQR problem for the linear system $(A(x_t), B(x_t))$ along the solution of \eqref{n-close-sys}. As will be observed, the  SDP discussed in Proposition \ref{riccati-sdp-lq-property} 
forms the basis of such a controller design.
We first specialize the problem to the scalar case to illustrate the findings.

\section{Scalar system control design}\label{scalar-sec}

Consider the nonlinear system \eqref{n-sys} with $x_t\in \R$. This gives rise to the scalar nonlinear system 
\begin{align}
    x_{t+1} = a(x_t)x_t + b(x_t)u_t, \label{scalar-state-dependent-form}
\end{align}
with $x \in \R$ the state, $u \in \R$ the control input, $x_0 \in \R$ the initial condition, $t \in \mathbb N_0$; $a: \R \rightarrow \R$ and $b: \R \rightarrow \R$ are continuous functions. 
We introduce the following assumption:   
\begin{assum}\label{assum-bx-bbarbelow}
    There exist a compact set $\mathcal{X} \subseteq \R$ including the origin and a positive constant $\underline b$ 
    such that $\eucnorm{b(x)} \geq \underline b$ for all $x \in \mathcal{X}$, where $b(x)$ is the input vector field defined in \eqref{scalar-state-dependent-form}. \qedp
\end{assum}

Note that under Assumption \ref{assum-bx-bbarbelow} and due to the continuity of $a(x)$, the state vector field is also bounded as 
\begin{equation}\label{e:L}
\eucnorm{a(x)} \leq L, \quad \forall x \in \mathcal{X}.
\end{equation}
for some $L>0$. To find the state-dependent control gain $k(x_t)$, we consider the SDP \eqref{lmi-lq} for system \eqref{scalar-state-dependent-form} with $Q$ and $R$ being identities, i.e.:
\begin{subequations}\label{sdlmi-lq}
\begin{align}
&\min\limits_{k(x), p(x), \ell(x), \gamma(x)} \quad \gamma(x) \label{cost}
\\
&\text{subject to} \notag
\\
 & (a(x)+b(x) k(x))^2 p(x) - p(x) + 1 \leq 0 \label{scalar-stab1}
\\ 
 & p(x) \geq 1 \label{scalar-stab2}
\\ 
& \ell(x) - p(x) k^2(x) \geq 0 \label{scalar-perf1} \\
& \gamma(x) \geq p(x) + \ell(x) \label{scalar-perf2}
\end{align}\label{scalar-lq}\end{subequations}%
with $x\in \mathcal{X}$. 
Note that, for any given $x\in \mathcal{X}$, the above program solves the LQ problem for the LTI system 
$x_{t+1} = (a(x) + b(x)k(x))x_t$.
The following result establishes feasibility of the program \eqref{scalar-lq} 
and a uniform upper bound on the cost  $\gamma(x)$ for every $x \in \mathcal{X}$.

\begin{prop}\label{thm-lq-scalar-bound}
For any $x\in \R$ such that $b(x) \neq 0$, the program \eqref{scalar-lq} is feasible and has a unique optimal solution  $(k(x), p(x),  \ell(x), \gamma(x))$. Moreover, under Assumption \ref{assum-bx-bbarbelow}, $\gamma(x)$ is uniformly upper bounded as
\begin{align}\label{ori-gamma-scalar-bound} 
 1 \leq \gamma(x)\leq \bar \gamma, \quad \forall x\in \mathcal{X},
\end{align}
where 
\begin{align}\label{gamma-scalar-bound} 
    \bar \gamma \coloneqq 1 + \frac{L^2}{\underline b^2}
\end{align}
with the scalar $L$ given by \eqref{e:L}. \qeds
\end{prop}

\emph{Proof}. See Appendix \ref{pf:prop:bargamma-analytic-scalar}. 
\qedp

\par
Next, we analyze stability of the nonlinear system 
\begin{align}\label{scalar-closed}
    x_{t+1} = (a(x_t) + b(x_t) k(x_t)) x_t, 
\end{align}
where the feedback controller $k(x_t)$ is obtained by recursively solving the program \eqref{sdlmi-lq} along the solution of the system. 
We recall the following definition. 
\begin{definition}
Let the origin be an asymptotically stable equilibrium of the system $x_{t+1} = f(x_t)$. 
A set $\mathcal{R}$ defines an RoA 
if for every $x_0 \in \mathcal{R}$ we have $\lim\limits_{t \rightarrow \infty} x_t = 0$. 
\end{definition}

For an arbitrarily chosen $r > 0$, let
\begin{align}\label{ball-def}
    B_r \coloneqq \{x\in \mathbb{R}: \eucnorm{x} \leq r\}.
\end{align}
In addition, let $\bar r$ be the largest value of $r$ such that $B_r\subseteq \mathcal{X}$. %
We then have the following result.     
\begin{thm}\label{RoA-lyap-scalar}
Let Assumption \ref{assum-bx-bbarbelow} hold. 
For each $t$, let $(k(x_t), p(x_t), \ell(x_t), \gamma(x_t))$ 
be the unique optimal solution to \eqref{scalar-lq} with $x=x_t$. Let $\bar \gamma$ be given by \eqref{gamma-scalar-bound}.
Then, any solution of closed-loop system \eqref{scalar-closed} initialized in 
\begin{align}\label{def-B-delta}
B_\delta = \{ x\in \mathbb R : \eucnorm{x} \leq \delta := \bar \gamma^{-1/2} \bar r\}
\end{align}
geometrically converges to the origin.  
Moreover, a Lyapunov function certifying stability in the sense specified below
\begin{align}\label{lyap-geo-dec-scalar}
    V_t(x_t) \leq (1-\bar \gamma^{-1})^t V_0(x_0).
\end{align}  
is given by
\begin{align}\label{lyap-def-scalar}
    V_t(x_t) = x^\top_t \tilde p^{-1}(x_t) x_t,
\end{align}
where 
\begin{align}\label{p-def-scalar}
    \tilde p(x_t) = \max\{p(x_t), \tilde p(x_{t-1})\}, t \geq 1
\end{align}
and $\tilde p(x_0) = p(x_0)$. \qeds
\end{thm}

\emph{Proof}:
Let $x_0 \in B_\delta \setminus \{0\}$. By Proposition \ref{thm-lq-scalar-bound}, the program \eqref{scalar-lq} is feasible with $x = x_0$. The optimal solution $(k(x_0), p(x_0), \ell(x_0), \gamma(x_0))$ satisfies \eqref{scalar-stab1}-\eqref{scalar-perf2}. 
Let 
\begin{align}\label{lyap-sequence}
    V_i(x_j) \coloneqq x_j^\top \tilde p^{-1}(x_i) x_j
\end{align}
with $i, j \in \mathbb N_0$. 
By multiplying both sides of \eqref{scalar-stab1} with $\tilde p^{-2}(x_0) x_0^2$, we have
\begin{align}\label{drop-pf1}
\begin{aligned}
     &\tilde p^{-1}(x_0) (a(x_0)+b(x_0)k(x_0))^2 x_0^2 - \tilde p^{-1}(x_0) x_0^2 \\
  \leq &-\tilde p^{-2}(x_0) x_0^2
\end{aligned}
\end{align}
namely, $V_0(x_1) \leq (1-\tilde p^{-1}(x_0)) V_0(x_0).$ 
Next, since $\tilde p(x_0) = p(x_0)$ where $p(x_0)$ satisfies
\begin{align}\label{scalar-p0-bound}
 1 \stackrel{\eqref{scalar-stab2}}{\leq} p(x_0)  \stackrel{\eqref{p-lq-bound-1}}{\leq}  \gamma(x_0) 
 \stackrel{\eqref{gamma-scalar-bound}}{\leq}  \bar \gamma   
\end{align}
we have
\begin{align}\label{drop-pf2}
\begin{array}{rl}
V_0(x_1) {\leq}  (1- \bar\gamma ^{-1}) V_0(x_0). 
\end{array}
\end{align}
By recalling the definition of $V_0(x_1)$ and applying the inequality \eqref{drop-pf2}, we find that %
\begin{align}\label{x1-x0-scalar-pf}
 \begin{array}{rl}
  \eucnorm{x_1}^2  \stackrel{\eqref{lyap-sequence}}{=}
 \tilde p(x_0) V_0(x_1)  \stackrel{\eqref{scalar-p0-bound} \eqref{drop-pf2}}{\leq} \bar \gamma (1-\bar\gamma ^{-1}) V_0(x_0).
 \end{array}
\end{align}
Moreover, since $V_0(x_0) \stackrel{\eqref{lyap-sequence}}{=} \tilde p^{-1}(x_0) \eucnorm{x_0}^2 
= p^{-1}(x_0) \eucnorm{x_0}^2 \stackrel{\eqref{scalar-p0-bound}}{\leq} \eucnorm{x_0}^2$ and $\eucnorm{x_0}^2 \leq \delta^2$ with $\delta = \frac{\bar r}{\sqrt{\bar \gamma}}$, it follows that 
\begin{align}\label{x1-in-ball}
    \eucnorm{x_1}^2 \stackrel{\eqref{x1-x0-scalar-pf}}{\leq} 
    \bar \gamma(1-\bar \gamma^{-1}) V_0(x_0) 
    \leq \bar \gamma (1-\bar \gamma^{-1}) \frac{\bar r^2}{\bar \gamma} .
\end{align}
Since by \eqref{gamma-scalar-bound}, $\bar \gamma \geq 1$, then we have $\eucnorm{x_1}^2 < \bar r^2$, which implies $x_1 \in B_{\bar r}$. 
\par
Next, since $B_{\bar r} \subseteq \mathcal{X}$, we have $x_1 \in \mathcal{X}$. 
If $x_1 = 0$, then it implies that the solution of the closed-loop system is at the equilibrium and the control input $u_t$ for $t \geq 1$ will be 0.  
If $x_1 \in B_{\bar r} \setminus \{0\}$, then we continue using \eqref{scalar-lq} to find the optimal controller $k(x_1)$ and the optimal solution $(p(x_1), \ell(x_1), \gamma(x_1))$, as analyzed in the proof of Proposition \ref{thm-lq-scalar-bound}. 
Moreover, since $p(x_1) \stackrel{\eqref{p-lq-bound-1}}{\leq} \gamma(x_1) \stackrel{\eqref{gamma-scalar-bound}}{\leq} \bar \gamma$ and $\tilde p(x_0) = p(x_0) \stackrel{\eqref{scalar-p0-bound}}{\leq} \bar \gamma$, it follows that $\max\{p(x_1), \tilde p(x_0)\} \leq \bar \gamma$. Recalling the definition of $\tilde p(x_1)$ in \eqref{p-def-scalar}, we have that 
\begin{align}\label{scalar-p-x1-def}
  1 \leq \tilde p(x_1) = \max\{p(x_1), \tilde p(x_0)\} \leq \bar \gamma.
\end{align}
\par
Next, by repeating the arguments in \eqref{drop-pf1} with $x_0$ replaced by $x_1$, we again have \eqref{drop-pf2} but one-step forward, namely, $V_1(x_2) \leq (1-\bar \gamma^{-1}) V_1(x_1).$ 
Moreover, since $\tilde p(x_1) \geq \tilde p(x_0)$, it gives
$V_1(x_1) \leq V_0(x_1)$. Hence, we have  
$V_1(x_2) \leq (1-\bar \gamma^{-1}) V_0(x_1) \stackrel{\eqref{drop-pf2}}{\leq} 
    (1-\bar \gamma^{-1})^2 V_0(x_0). $     
This enables us to repeat the arguments in \eqref{x1-x0-scalar-pf}, \eqref{x1-in-ball} and to have $\eucnorm{x_2}^2 < \bar r^2$. 
By recursively repeating these arguments, it holds that $\eucnorm{x_t}^2 < \bar r^2$ and thus $x_t \in B_{\bar r} \subseteq \mathcal{X}$. 
Then, by Proposition \ref{thm-lq-scalar-bound}, feasibility of the program \eqref{sdlmi-lq} along with the solution $x_t$ of the closed-loop system \eqref{scalar-closed} is guaranteed for every $t \in \mathbb N_0$. 
\par
Next, notice that for every $t \in \mathbb N_0$, the inequalities $V_t(x_{t+1}) \leq (1-\bar \gamma^{-1}) V_t(x_t) $ and $V_{t+1}(x_{t+1}) \leq V_t(x_{t+1})$ hold. Then, we have $V_{t+1}(x_{t+1}) \leq (1-\bar \gamma^{-1}) V_t(x_t)$ 
and thus \eqref{lyap-geo-dec-scalar} holds. 
Moreover, note that \eqref{scalar-p-x1-def} holds with $x_0$ replaced by $x_t$ and $x_1$ replaced by $x_{t+1}$ for any $t \in \mathbb N_0$, and thus 
\begin{align}\label{tilde-p-bound}
  1 \leq \tilde{p}(x_t) \leq \bar \gamma, \quad t \in \mathbb N_0.
\end{align}
Then, it holds that 
\begin{align}\label{xk-x0-scalar}
\begin{array}{rl}
    \eucnorm{x_t}^2 \stackrel{\eqref{lyap-def-scalar}}{=} \tilde p(x_t) V_t(x_t) \stackrel{\eqref{tilde-p-bound}}{\leq} \bar\gamma V_t(x_t)     \stackrel{\eqref{lyap-geo-dec-scalar}}{\leq}
    \bar\gamma (1-\bar \gamma^{-1})^t V_0(x_0) .
\end{array}
\end{align}
By \eqref{tilde-p-bound}, $\tilde p^{-1}(x_0)$ satisfies $\bar \gamma^{-1} \leq \tilde p^{-1}(x_0) \leq 1$, which implies $V_0(x_0) \leq \eucnorm{x_0}^2$. Then by \eqref{xk-x0-scalar}, we have $ \eucnorm{x_t}^2    \leq \bar\gamma (1-\bar \gamma^{-1})^t 
    \eucnorm{x_0}^2$. 
This completes the proof.    \qedp

The result shows that under Assumption 1, by recursively solving \eqref{scalar-lq}, the solution of the closed-loop system geometrically converges to the origin provided that the initial condition is inside the set $B_\delta$ defined as in \eqref{def-B-delta}. As a result, the set $B_\delta$ defines an RoA estimate for system \eqref{scalar-closed}. 
Moreover, notice that the definition of $\tilde p(x)$ in \eqref{p-def-scalar} ensures the sequence $V_i(x_j)$ defined as in \eqref{lyap-sequence} is monotonically decreasing, as illustrated in Fig.~\ref{fig-lyap-drop-illu}. 

\begin{figure}[htb]
    \centering
    \includegraphics[scale=0.4]{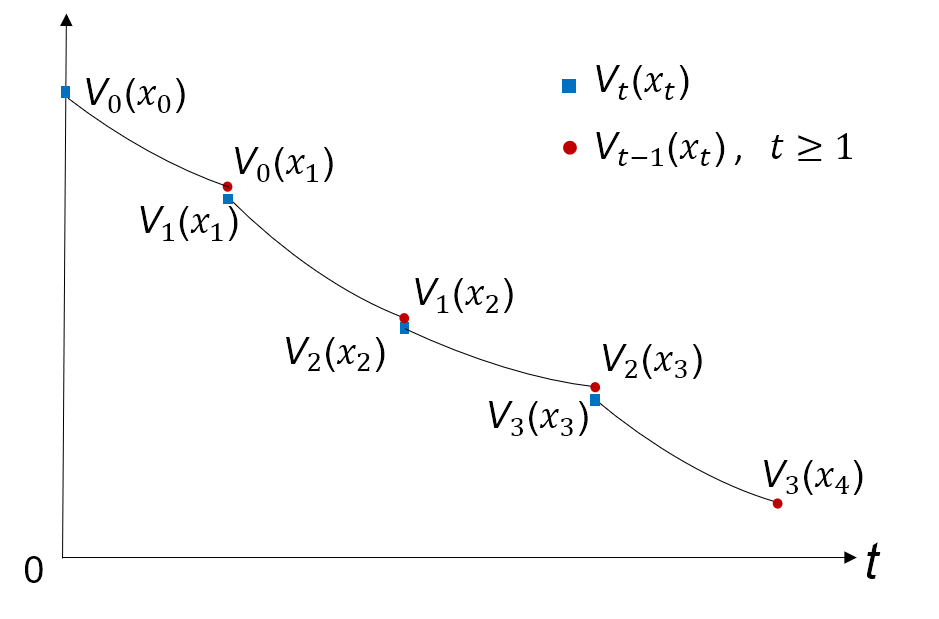}
    \caption{The illustration of the monotonically decreasing sequence. }
    \label{fig-lyap-drop-illu}
\end{figure}

\section{Higher-dimensional system control design}\label{higher-dimen-sec}

We continue with the control design for the higher-dimensional system described in \eqref{n-sys}.  
To find the sequence $K(x_t)$ in \eqref{general-controller}, we generalize \eqref{scalar-lq} to  
\begin{subequations}\label{n-lq}
   \begin{align}
&\min\limits_{K(x), P(x), L(x), \gamma(x)} \quad \gamma(x) \label{n-cost}
\\
&\text{subject to} \notag
\\ & 
(A(x)+B(x) K(x)) I \cdot P(x) (\star)^\top - P(x) + I_n \preceq 0     
\label{n-dimension-stab1}
\\ & P(x) \succeq I_n \label{n-dimension-stab2}
\\ & L(x) - K(x) P(x) K^\top(x) \succeq 0 \label{n-dimension-perf1} 
\\ & \gamma(x) \geq \operatorname{tr} P(x) + \operatorname{tr} L(x) \label{n-dimension-perf2}
\end{align} 
\end{subequations}%
where $K(x) \in \R^{m \times n}$, $P(x) \in \mathbb S^{n\times n}$, $L(x) \in \mathbb S^{m \times m}$, $\gamma(x) \in \R$. 
Then, for any $x_t \in \R^n$ that makes $(A(x_t), B(x_t))$ stabilizable, it follows from Proposition \ref{riccati-sdp-lq-property} that there exists a unique optimal solution $(K(x_t), P(x_t), L(x_t), \gamma(x_t))$ to \eqref{n-lq} with $x = x_t$.

We revisit Theorem \ref{RoA-lyap-scalar} and use \eqref{n-lq} as a basis to design another SDP whose optimal solutions provide us a sequence of feedback control gain $K(x_t)$ in \eqref{n-close-sys} and a matrix sequence $P(x_t) \succ 0$ that builds a time-varying Lyapunov function 
\begin{align} \label{lyap-def-n-dim}
    V_t(x_t) =
     x_t^\top P^{-1}(x_t) x_t 
\end{align}
to certify the stability of the origin of \eqref{n-close-sys}. 
To this end, inspired by the Lyapunov stability criterion for discrete-time time-varying systems \cite[Thm. 23.3]{rugh1996linear}, we impose the following Lyapunov stability conditions: 
\begin{subequations}\label{conditions}
   \begin{align}
& \bar \gamma^{-1} \eucnorm{x_t}^2 \leq V_t (x_t) \leq \eucnorm{x_t}^2  \label{cond1} \\
& V_{t}(x_{t+1}):= x_{t+1}^\top P^{-1}(x_{t}) x_{t+1} < V_t(x_t) \label{asym-cond2} \\
& V_{t+1}(x_{t+1}) \leq V_t(x_{t+1}) \label{cond3}
\end{align} 
\end{subequations}%
where $\bar \gamma$ is a positive constant to be defined later. 
For the case of geometric stability, \eqref{asym-cond2} is strengthened to 
\begin{align}\label{cond2}
V_{t}(x_{t+1}) \leq (1-\bar \gamma^{-1}) V_t(x_t).
\end{align}

\subsection{Time-varying SDP-based online nonlinear control} \label{pf-reform}

From $t \geq 1$, we modify \eqref{n-lq} by setting $x= x_t$ and adding the constraint
\begin{align}
    \begin{bmatrix} x^\top_t P^{-1}(x_{t-1}) x_t & x_t^\top \\ * & P(x_t) \end{bmatrix} \succeq 0, 
\end{align}  
and a user-defined constant $\bar \gamma$ that serves as a uniform upper bound for the sequence $\gamma(x_t)$. 
The solution to the modified SDP is established in the following proposition. 

\begin{prop}\label{redef-sdp-property}
Consider the program:  
\begin{subequations}\label{n-sdp}
\begin{align}
&\min\limits_{Y(x_t), P(x_t), L(x_t), \gamma(x_t)} \quad \gamma(x_t) \label{n-sdp-cost}
\\
&\text{subject to} \notag
\\ & \begin{bmatrix} P(x_t)-I & A(x_t) P(x_t) + B(x_t) Y(x_t) \\ * & P(x_t)
\end{bmatrix} \succeq 0 \label{n-sdp-stab1}
\\ & \begin{bmatrix} L(x_t) & Y(x_t)\\ * & P(x_t)
\end{bmatrix} \succeq 0 \label{n-sdp-perf1} 
\\ & \gamma(x_t) \geq \operatorname{tr} P(x_t) + \operatorname{tr} L(x_t) \label{n-sdp-perf2}
\\ & \begin{bmatrix} x^\top_t P^{-1}(x_{t-1}) x_t & x_t^\top \\ * & P(x_t) \end{bmatrix} \succeq 0, t \geq 1 
\label{n-sdp-stab3}
\\ & \gamma(x_t) \leq \bar \gamma , t \geq 1 
\label{n-sdp-stab4}
\end{align}\end{subequations}%
with $t \geq 1$, 
$Y(x_t) \in \R^{m \times n}$, $P(x_t) \in \mathbb S^{n \times n}$, $L(x_t) \in \mathbb S^{m \times m}$ and $\gamma(x_t) \in \R$. Suppose \eqref{n-sdp} is feasible for a given point $x_t\in \R^n$ and a given matrix 
$P(x_{t-1}) \succ 0$. Then there exists an optimal solution to \eqref{n-sdp}. Moreover, any optimal solution $(Y(x_t), P(x_t), L(x_t), \gamma(x_t))$ is such that 
\begin{align}\label{opt-sol-bound-P} 
  I \preceq P(x_t) \preceq \bar \gamma I 
\end{align}
and
\begin{align}\label{decre-lyap} \operatorname{L_{cl}}(x_t) := A_{cl}^\top(x_t) P^{-1}(x_t) A_{cl}(x_t) - P^{-1}(x_t) \prec 0
\end{align}
where $A_{cl}(x_t)$ is given by \eqref{n-close-sys} with $K(x_t):= Y(x_t) P^{-1}(x_t)$.
Moreover, if $P(x_t)-I$ is nonsingular, then 
\begin{align}\label{decre-geo-lyap} 
\operatorname{L_{cl}}(x_t) \preceq - \lambda_m(P^{-1}(x_t)) P^{-1}(x_t) 
\end{align}
\end{prop}

\emph{Proof}. See Appendix \ref{pf:prop:regu-sdp}. \qedp  

Note that at $t = 0$, \eqref{n-sdp} boils down to \eqref{n-lq} with $x=x_0$ and thus by Proposition \ref{riccati-sdp-lq-property}, stabilizability of $(A(x_0), B(x_0))$ implies feasibility and existence of a unique optimal solution $(Y(x_0), P(x_0), L(x_0), \gamma(x_0))$.
For $t \geq 1$, by assuming feasibility of \eqref{n-sdp} at $t-1$, \eqref{n-sdp-stab3} uses an optimal solution $P(x_{t-1})$ as the given positive definite matrix and thus \eqref{n-sdp} is well-defined. Moreover, in doing so, we have the following stability result. 

\begin{thm}\label{n-dim-stability}
Assume that \eqref{n-sdp} is feasible for every $t \geq 0$ and a given $\bar \gamma$. Then, the origin of the closed-loop system \eqref{n-close-sys} with $K(x_t):= Y(x_t) P^{-1}(x_t)$ and $(Y(x_t), P(x_t), L(x_t), \gamma(x_t))$ any optimal solution to \eqref{n-sdp} is asymptotically stable. 
Furthermore, if the matrix $P(x_t)-I$ is nonsingular  for all $t$,  
then the origin is geometrically stable, namely 
\begin{align}\label{n-dim-geo}
    \eucnorm{x_t} \leq \sqrt{\bar \gamma} \sqrt{(1- \bar \gamma^{-1})^t} \eucnorm{x_0}.
\end{align}
\end{thm}

\emph{Proof}:
To prove asymptotic stability of the origin, we show that \eqref{lyap-def-n-dim} satisfy \eqref{conditions} and thus serves as a Lyapunov function for the closed-loop system \eqref{n-close-sys}. 
By \eqref{opt-sol-bound-P}, \eqref{cond1} holds.   
By \eqref{decre-lyap}, the inequality \eqref{asym-cond2} holds. Moreover, computing the Schur complement of \eqref{n-sdp-stab3} yields \eqref{cond3}. 
This completes the proof for asymptotic stability of the origin. 
\par
By \eqref{opt-sol-bound-P} and under the nonsingulairty condition on $P(x_t)-I$, we have $I \prec P(x_t) \preceq \bar \gamma I$. 
Then, by Proposition \ref{redef-sdp-property}, \eqref{decre-geo-lyap} holds and $\lambda_m (P^{-1}(x_t)) = \lambda_M^{-1} (P(x_t)) \geq \bar \gamma^{-1}$. Hence, we have $\operatorname{L}_{cl}(x_t) \preceq -\bar \gamma^{-1} P^{-1}(x_t)$  
and thus \eqref{cond2} holds, which together with \eqref{cond1} and \eqref{cond3}, imply geometric stability of the origin as given by \eqref{n-dim-geo}. 
\qedp 

A few remarks are in order. 
\par
\emph{On the initialization of the time-varying controller.} 
At time $t=0$, the constraints \eqref{n-sdp-stab3} and \eqref{n-sdp-stab4} are absent from the SDP \eqref{n-sdp}. Hence, the initial controller gain is the same as the one obtained  
by solving \eqref{n-lq} with $x=x_0$. 

\par
\emph{On the nonsingularity condition of $P(x_t) - I$. }
The additional nonsingularity condition in Theorem \ref{n-dim-stability} is equivalent to nonsingularity of the closed-loop matrix sequence $A(x_t) + B(x_t) K(x_t)$, which can be interpreted as the function mapping $x_t$ to $x_{t+1}$ is bijective   
and such that the uniqueness of solutions backward in time holds for the closed-loop system \eqref{n-close-sys} \cite[Sec. 13.2]{haddad2011nonlinear}.  
\qeds

\subsection{Recursive feasibility}\label{rec-feas-sec}

Theorem \ref{n-dim-stability} establishes stability of the SDP-based control scheme under the assumption that \eqref{n-sdp} remains feasible at each time step. Exploiting the temporal coupling introduced by \eqref{n-sdp-stab3}, we derive conditions under which feasibility at time step $t-1$ guarantees feasibility at time step $t$ for $t \ge 1$.

\subsubsection{Condition 1}

One possibility for studying recursive feasibility is to frame the changes in the dynamics of \eqref{n-sys} as consecutive deviations of the pair $(A(x_t), B(x_t))$ from the nominal dynamics $(A(x_{t-1}), B(x_{t-1}))$. 
Intuitively, for a sufficiently small deviation $\epsilon \geq 0$ with
\begin{align}\label{epsidef}
 \begin{array}{rl}
 \max\{ \specnorm{A(x_t) - A(x_{t-1})}, \specnorm{B(x_t) - B(x_{t-1})}\} \leq \epsilon
   \end{array}
\end{align}
the controller $K(x_{t-1})$ that stabilizes $(A(x_{t-1}), B(x_{t-1}))$ can stabilize $(A(x_{t}), B(x_{t}))$. 
Drawing inspiration from the certainty equivalence-based control design \cite{de2021low}, we introduce some additional notation and state the result below.

\par
For a given $t \in \mathbb N_0$ and a feasible solution $(Y(x_{t}), P(x_{t}), L(x_{t}), \gamma(x_{t}))$ to \eqref{n-sdp}, define
\begin{align}\label{bre-notation}
  \begin{array}{l}
    Z_{t} := \begin{bmatrix}
     A(x_{t}) & B(x_{t})
  \end{bmatrix}, R_{t} \coloneqq Z_{t+1} - Z_{t}, \\[3pt]
  M_{t} := \begin{bmatrix} P(x_{t}) \\ Y(x_{t}) \end{bmatrix} P^{-1}(x_{t}) \begin{bmatrix} P(x_{t}) \\ Y(x_{t}) \end{bmatrix}^\top
  \\[10pt] 
\Theta_{t} \coloneqq Z_{t} M_{t} Z_{t}^\top - P(x_{t})\\\Psi_{t}:=
R_{t} M_{t} R_{t}^\top + Z_{t}M_{t}R_{t}^\top + R_{t} M_{t} Z_{t}^\top .
    \end{array}
\end{align}
\begin{prop}\label{rec-fea-2}
Suppose \eqref{n-sdp} is feasible at time instant $t-1$, and let $(Y(x_{t-1}), P(x_{t-1}), L(x_{t-1}), \gamma(x_{t-1}))$ be its corresponding optimal solution. 
If
    \begin{align}\label{cecondition}
    \Psi_{t-1} \preceq (1- \frac{\gamma(x_{t-1})}{\bar \gamma}) I, 
\end{align}
then the solution $\eta_{t-1} (Y(x_{t-1}), P(x_{t-1}), L(x_{t-1}), \gamma(x_{t-1}))$ with 
$$\eta_{t-1} := \frac{\bar \gamma}{\gamma(x_{t-1})}$$
 is feasible for \eqref{n-sdp}.
\qeds
\end{prop}

{\it{Proof.}} Since $(Y(x_{t-1}), P(x_{t-1}), L(x_{t-1}), \gamma(x_{t-1}))$ is feasible for \eqref{n-sdp} at the time step $t-1$, then it follows from \eqref{n-sdp-stab1} and the definition of $\Theta_{t}$ in \eqref{bre-notation} that $\Theta_{t-1} + I \preceq 0$. Hence, 
\begin{align}\label{pf-us-les}
\begin{array}{rl}
    &\eta_{t-1} \Theta_{t-1} + \eta_{t-1} \Psi_{t-1} + I \\
    = & \eta_{t-1} (\Theta_{t-1} + \Psi_{t-1}) + \eta_{t-1} I + (1-\eta_{t-1}) I \\
    = &\eta_{t-1} (\Theta_{t-1} + I) + (\eta_{t-1} \Psi_{t-1} + (1-\eta_{t-1}) I) \preceq 0    
\end{array}
\end{align}
where the inequality follows from $\eta_{t-1} (\Theta_{t-1} + I) \preceq 0$ and \eqref{cecondition}. Hence $\eta_{t-1} (Y(x_{t-1}), P(x_{t-1}), L(x_{t-1}), \gamma(x_{t-1}))$ satisfies \eqref{n-sdp-stab1} at time step $t$. Moreover, by construction, $\eta_{t-1} (Y(x_{t-1}), P(x_{t-1}), L(x_{t-1}), \gamma(x_{t-1}))$ with $\eta_{t-1} \geq 1$ satisfies all the other constraints of \eqref{n-sdp}. \qedp

Proposition \ref{rec-fea-2} provides a verifiable condition for the feasibility of \eqref{n-sdp}. Yet, this condition only guarantees feasibility one step ahead. We therefore pursue a stronger condition that ensures feasibility for all time.

\subsubsection{Condition 2}

We specialize the system \eqref{n-sys} as 
\[
x_{t+1} = A(x_t) x_t + B u_t
\]
to illustrate the finding. 
Let $A(x)$ be partitioned as 
\[
A(x) = \bar A + \hat{A}(x)
\]
with $\hat{A}: \R^n \rightarrow \R^{n \times n}$ a matrix-valued function. 
Given the continuity of $A(x)$ \cite{cloutier1999recoverabilityC1}, the function $\hat{A}(x) = A(x) - \bar{A}$ is also continuous and thus bounded on any compact set $\mathcal{X} \subseteq \R^n$. 
We have the following. 
\begin{prop}\label{prop:offline:lmi:rf}
Let $\Delta \in \R^{n \times r}$ be a known matrix and $\mathcal{X} \subseteq \R^n$ a set including the origin such that 
\begin{align}\label{hatA-strength-condition} 
\hat{A}(x) \hat{A}(x)^\top \preceq \Delta \Delta^\top, \forall x \in \mathcal{X}    
\end{align}
Consider the following LMI in the decision variables $\epsilon > 0$, $P \in \mathbb{S}^{n\times n}$, and $Y \in \R^{m \times n}$: 
\begin{align}\label{converge-lmi-state-independent-B}
\begin{bmatrix}
    P - I - \epsilon \Delta \Delta^\top & \bar{A} P + B Y & 0 \\ 
    (\bar{A} P + B Y)^\top & P & P \\ 
    0 & P & \epsilon I \end{bmatrix} \succeq 0 .
\end{align}
Suppose \eqref{converge-lmi-state-independent-B} is feasible and let $(\epsilon, P, Y)$ be any feasible solution. 
Then, for any $x_0 \in \mathcal{R}_\alpha := \{x\in \R^n: x^\top P^{-1} x \le \alpha, \alpha > 0\}$ contained in $\mathcal{X}$, the time-varying SDP \eqref{n-sdp} with  
\begin{align}
    \label{eq:set:bargamma:model}
\bar \gamma \ge \operatorname{tr}(P) + 
\operatorname{tr}(Y P^{-1} Y^\top)
\end{align}
is feasible for all $t \ge 0$. \qeds
\end{prop}

{\it{Proof.}} By Schur complement, \eqref{converge-lmi-state-independent-B} implies 
\[
\begin{array}{rl}
&\begin{bmatrix}
    P - I - \epsilon \Delta \Delta^\top & \bar{A} P + BY \\ (\bar{A} P + BY)^\top & P
\end{bmatrix} - \epsilon^{-1}
\begin{bmatrix} 0 \\ P \end{bmatrix} \begin{bmatrix} 0 & P \end{bmatrix} 
\\
= 
&\begin{bmatrix}
    P - I & \bar{A} P + BY \\ (\bar{A} P + BY)^\top & P
\end{bmatrix} - \epsilon \begin{bmatrix}
    -I \\ 0 
\end{bmatrix} \Delta \Delta^\top
\begin{bmatrix}
    -I & 0 
\end{bmatrix} 
\\
& - \epsilon^{-1}
\begin{bmatrix} 0 \\ P \end{bmatrix} \begin{bmatrix} 0 & P \end{bmatrix} \succeq 0.
\end{array}
\]
Then, by the nonstrict Petersen's lemma \cite[Fact 2]{bisoffi2022pertersen}, for any matrix $\hat{A}(x)$ satisfying \eqref{hatA-strength-condition}, it holds that 
\[
\begin{array}{rl}
\begin{bmatrix}
    P - I & \bar{A} P + BY \\ (\bar{A} P + BY)^\top & P
\end{bmatrix} - \begin{bmatrix} 0 \\ P
\end{bmatrix} \hat{A}^\top(x) \begin{bmatrix}
    -I & 0 
\end{bmatrix} \\
- \begin{bmatrix}
    -I \\ 0
\end{bmatrix} \hat{A}(x) \begin{bmatrix} 0 
& P \end{bmatrix} \succeq 0\end{array}
\]
and thus 
\[
\begin{bmatrix}
    P - I & \bar{A} P + \hat{A}(x) P + BY \\ (\bar{A} P + \hat{A}(x) P + BY)^\top & P
\end{bmatrix} \succeq 0
\]
By another use of Schur complement, we obtain $P - I - (\bar{A} P + \hat{A}(x) P + BY) P^{-1} (\bar{A} P + \hat{A}(x) P + BY)^\top \succeq 0$. 
Let $K := Y P^{-1}$. It holds that 
$P - I - (\bar{A} + \hat{A}(x) + BK) P (\bar{A} + \hat{A}(x) + BK)^\top \succeq 0, \forall x \in \mathcal{X}.$
By retracing the same steps as in the proof of Proposition \ref{redef-sdp-property} with $A_{cl}(x_t)$ and $P(x_t)$ replaced by $\bar A+\hat{A}(x) + BK$ and $P$, the inequality \eqref{decre-lyap} holds. 
Then, the solution to the closed-loop system 
$ x^+ = (A(x) + BK) x = (\bar A + \hat{A}(x) + BK)x$
with $x_0 \in \mathcal{R}_\alpha$ contained in $\mathcal{X}$, 
satisfies $V(x^+) - V(x) < 0, \forall x \ne 0$
with $V(x) = x^\top P^{-1} x$. 
Note that the solution $Y(x_t) \equiv KP$, $P(x_t) \equiv P$, $L(x_t) \equiv KPK^\top$, and $\gamma(x_t) \equiv \operatorname{tr}(P) + \operatorname{tr}(KPK^\top)$ is feasible to the time-varying SDP \eqref{converge-lmi-state-independent-B} with $x = x_t$ for all $t \ge 0$. The thesis is proven.  
\qedp

Proposition \ref{prop:offline:lmi:rf} presents an offline LMI whose solution yields RoA estimates for the closed-loop system
\[
x_{t+1} = (A(x_t) + B K(x_t) ) x_t
\]
where $K(x_t) := Y(x_t) P(x_t)^{-1}$ and $(Y(x_t), P(x_t))$ are optimal solutions for \eqref{n-sdp}. 
Proposition \ref{prop:offline:lmi:rf} thus complements Theorem \ref{thm-lq-scalar-bound} by providing a computable RoA estimate from which convergence to the origin is guaranteed. 
Proposition \ref{prop:offline:lmi:rf} further extends to systems with state dependent input vector fields $B(x)$, see Appendix \ref{appdx:rf:Bxcase}.  
\qedp

\section{Case study}\label{case-sec}

In this section, three systems with different types of nonlinear dynamics are presented to show the effectiveness of the proposed optimization-based online control method. 

\subsection{Single-machine infinite-bus power system}\label{case-SMIB}

A classical scenario in the study of the synchronous machine (SM) based power systems is the single-machine infinite bus (SMIB) \cite{kundorbook}. The analysis and (learning-based) control design of SMIB is fundamental for SM-based power systems \cite{barabanov2017conditions, verrelli2021nonlinear, jiang2014robust}. We use a third-order SM model \cite[eq. 5.11]{claudio2019energy} and include a power control loop to illustrate the effectiveness of the proposed online nonlinear control method. 
We consider 
\begin{align}
\begin{array}{rl}
     \dot \delta &= \Delta \omega \\
    M \dot{\Delta \omega} &= P_{m} - D \Delta \omega + B_s E_{q}^{\prime} V_s \sin \delta \\
    T_{do}^{\prime} \dot E_{q}^{\prime} &= E_f - E_{q}^{\prime} + (X_d - X_d^{\prime}) (B_s E_{q}^{\prime} - B_s V_s \cos \delta)    \\
    \dot P_m &= -\frac{1}{T_G} P_m - K_G \Delta \omega + v
\end{array}
\end{align}
with state variables 
\begin{itemize}
  \item[-] $\delta$: rotor angle with respect to (w.r.t.) synchronously rotating reference frame,  
  \item[-] $\Delta \omega$: frequency deviation w.r.t. synchronous frequency, 
  \item[-] $ E_q^{\prime} $: the $q$-axis component of the transient emf, 
  \item[-] $P_m$: mechanical input power, 
\end{itemize}
and system parameters 
\par
\begin{itemize}
  \item[-] $B_s$: susceptance of the transmission line between the SM and the infinite bus, 
 \item[-] $V_s$: voltage of infinite bus,  
 \item[-] $M$: moment of inertia, 
 \item[-] $D$: asynchronous damping constant, 
\item[-]  $T_{do}$: direct axis transient open-circuit constant, 
\item[-]  $X_d$: direct synchronous reactance,  
\item[-]  $X_d^{\prime}$: direct transient reactance, 
\item[-]  $E_f$: excitation voltage. 
\end{itemize}

Next, we denote the steady state as $(\bar \delta, 0, \bar E_{q}^{\prime}, \bar P)$ with $\bar v$ the constant input satisfying $\bar v = \frac{1}{T_G} \bar P$. 
Let 
\[
x:= \col(\delta- \bar \delta, \Delta \omega, E_{q}^{\prime} - \bar E_{q}^{\prime}, P_m - \bar P), \quad u := v - \bar v. 
\]
We refer to the $i$th state variable as $x(i)$. 
Then, by using the forward Euler discretization and a sampling time $T$, we have the discrete-time model in \eqref{n-sys} with 
\begin{align*}
\small{A(x) :=} \left[
 \begin{smallmatrix}
        1 & T & 0 & 0 \\[6pt]  
        b_1 T \frac{(\cos(\delta_{0} + \frac{x(1)}{2}) \sin(\frac{x(1)}{2}))}{x(1)} & 1-b_2 T & b_3 T \sin(x(1) + \delta_0) & b_4 T \\[6pt] 
        c_1 T \frac{\sin(\delta_0+\frac{x(1)}{2}) \sin(\frac{x(1)}{2})}{x(1)} & 0 & 1+ c_2 T  & 0 \\[6pt]  0 & -d_1 T & 0 & 1 -d_2 T  
 \end{smallmatrix}
\right]
\end{align*}
if $x(1) \neq 0$, and
\begin{align*}
    A(x) :=  \left[
 \begin{smallmatrix}
        1 & T & 0 & 0 \\ \frac{b_1 T \cos(\delta_0)}{2} & 1-b_2 T & b_3 T \sin(\delta_0) & b_4 T \\ \frac{c_1 T \sin(\delta_0)}{2}  & 0 & 1+T c_2 & 0 \\ 0 & -d_1 T & 0 & 1 -d_2 T  
    \end{smallmatrix}
\right]
\end{align*}
if $x(1) = 0$, where $b_1 = \frac{2 B_s}{M} \bar E_{q}^{\prime} V_s$, $b_2 = \frac{D}{M}$, $b_3 = \frac{B_s}{M} V_s$, $b_4 = \frac{1}{M}$, $c_1 = 2\frac{(X_d - X_d^{\prime})B_s V_s}{T_{do}^{\prime}}$, $c_2 = \frac{(X_d - X_d^{\prime})B_s - 1}{T_{do}^{\prime}}$, $d_1 = K_G$, $d_2 = \frac{1}{T_G}$.
The input vector field is $B := \col(0,0,0,T)$.

\par 
We set the parameters as follows: $X_d = 1.863$, $X_d^\prime = 0.257$, $B_{s} = -0.404$, $T_{do}^{\prime} = 0.5 s$, $\delta_0 = 0.236$, $\omega_0 = 50 Hz$, $D = 0.5$, $M = 1.25$,  
$E_{q0}^{\prime} = 1$, $V_s = 1$, $K_G = 1$, $T_G = 0.1$, $T = 0.1$. Moreover, we consider weighting matrices $Q = \diag(10, 8, 0.1, 0.1)$, $R = 1$. 
The initial condition of $x(1)$ is taken as a random variable within $[-2(\pi-\delta_0), 2(\pi-\delta_0)]$. The initial conditions of $x(2)$, $x(3)$, and $x(4)$ are taken as random variables within $[-1, 1]^3$. 
Solving \eqref{n-sdp} with $t = 0$, we obtain the optimal solution $\gamma(x_0)$ and $P(x_0)$. We then define $\bar \gamma := \frac{\lambda_M(P(x_0))}{\lambda_m(P(x_0))} \gamma(x_0)$. 

We consider 4 scenarios: open-loop (with $u_t = 0$), closed-loop with a time-invariant controller, closed-loop with the online controller that solves \eqref{n-lq} along the trajectory, and closed-loop with the online controller that solves \eqref{n-sdp} where \eqref{n-sdp-perf2} is modified to $\gamma(x_t) \geq \operatorname{tr}(QP(x_t))+ \operatorname{tr}(RL(x_t))$. 
The time-invariant controller used in the second scenario is the LQR controller for the linearized pair. The SDP \eqref{n-lq} and \eqref{n-sdp} are solved using MOSEK \cite{aps2019mosek}. 

Figure \ref{smis-x-large-1} displays the evolution of the state. To compare the performance quantitatively, we consider the performance index 
\begin{align}\label{perf-ind}
    J = \sum_{t=0}^{T_{end}-1} u_t^\top R u_t + x_t^\top Q x_t
\end{align}
with $T_{end} = 200$. The result is as follows: $J_1 = 955.1222$, $J_2 = 3.0122\times 10^{3}$, $J_3 = 703.0905$, $J_4 = 714.9147$, where the subscript $i$ of $J$ corresponds to the $i$th scenario. 
The online controller solving \eqref{n-sdp} achieves regulation to the origin with mild performance degradation relative to \eqref{n-lq}, while guaranteeing closed-loop stability.

\begin{figure}[ht!]
\centering
{{\includegraphics[width=6cm]{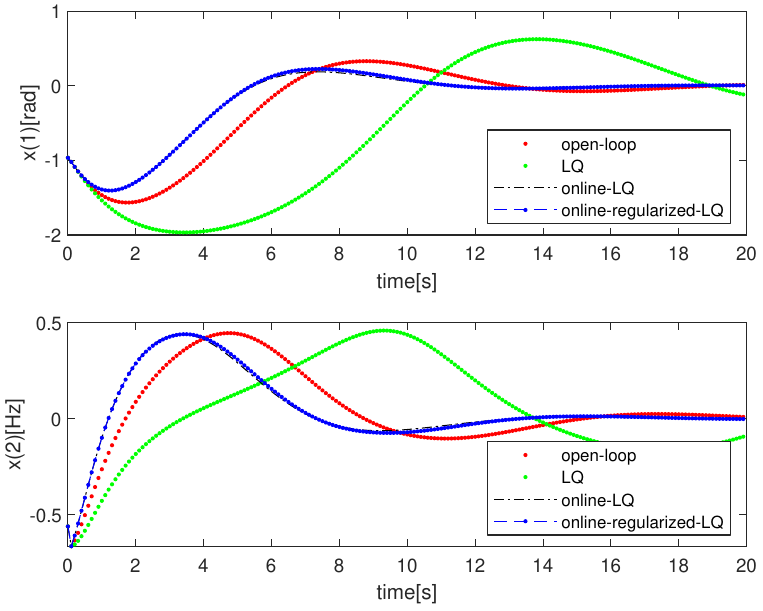} }}%
{{\includegraphics[width=6cm]{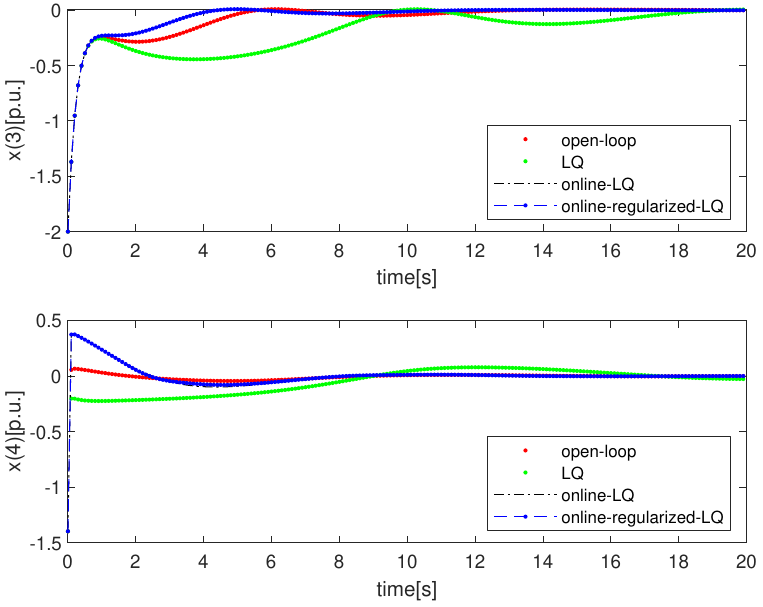} }}%
    \caption{The evolution of the state variables with the initial condition \\
    $x_0 = \col(-0.9644, -0.5594, -1.9998, -1.3953)$.}%
\label{smis-x-large-1}%
\end{figure}%

\subsection{Jet Engine}

As a representative polynomial system, we use the jet engine dynamics from \cite[eq. 2.208-2.210]{krstic1995nonlinear} to test our SDP-based online control scheme. The continuous system is discretized using forward Euler with sampling time $T$. 
Specifically, we have \eqref{n-sys} with 
\begin{align*}
   A(x) = \left[
    \begin{smallmatrix}
        1-\sigma T x(1) & - 2 \sigma T x(1) -\sigma T x(1) x(2)  & 0 \\ -3T & 1-3T x(1) - 1.5T x(2) - 0.5 T x^2(2) & -T \\ 
        0 & 0 & 1       
    \end{smallmatrix}\right]
\end{align*}
and $B = \col(0, 0, -T)$. 

\par
Let $\sigma = 3$, $T = 0.1$. 
The initial condition of $x(1)$, i.e. the mass flow, is taken as a random variable within $[0, 1]$. The initial conditions of $x(2)$ and $x(3)$ are taken as random variables within $[-1, 1]^2$. As in the first example, let $\bar \gamma := \frac{\lambda_M(P(x_0))}{\lambda_m(P(x_0))} \gamma(x_0)$ and apply the result in Theorem \ref{n-dim-stability} to design the controller. Note that the origin of the open-loop system is unstable. 
To compare the transient performance, we further simulate the evolution of the state of the closed-loop system \eqref{n-close-sys} where $(K(x_t), P(x_t))$ is the unique optimal solution for \eqref{n-lq} with $x=x_t$ and compute the value of $V_1(x_t) := x_t^\top P^{-1}(x_t) x_t$. 
\par
As shown in Fig.~\ref{jet-engine-state-evo}, both controllers regulate the state variables to the origin. Consider the performance index $J$ defined as in \eqref{perf-ind} with $T_{end} = 100$. The performance index for the system with the online controller solving \eqref{n-sdp} is obtained as $4.9483$, whereas that of the online controller solving \eqref{n-lq} is computed as $3.5569$. 
Regarding stability, the controller solving \eqref{n-sdp} provides a Lyapunov certificate $V_t(x_t)$ (defined in \eqref{lyap-def-n-dim}) that decreases \emph{monotonically} along the trajectory, while the controller solving \eqref{n-lq} yields $V_1(x_t)$ which does not exhibit monotonic decrease.

\begin{figure}[hbt]
\centering
{{\includegraphics[width=6cm]{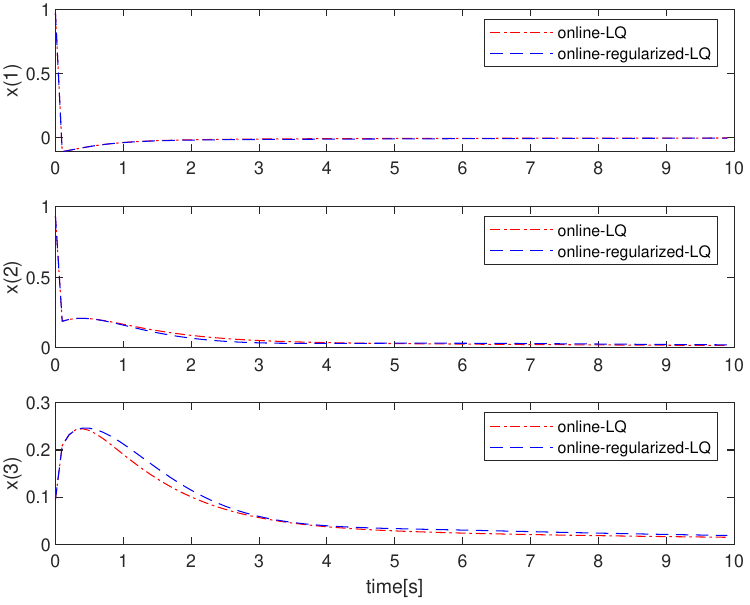} }}%
{{\includegraphics[width=6cm]{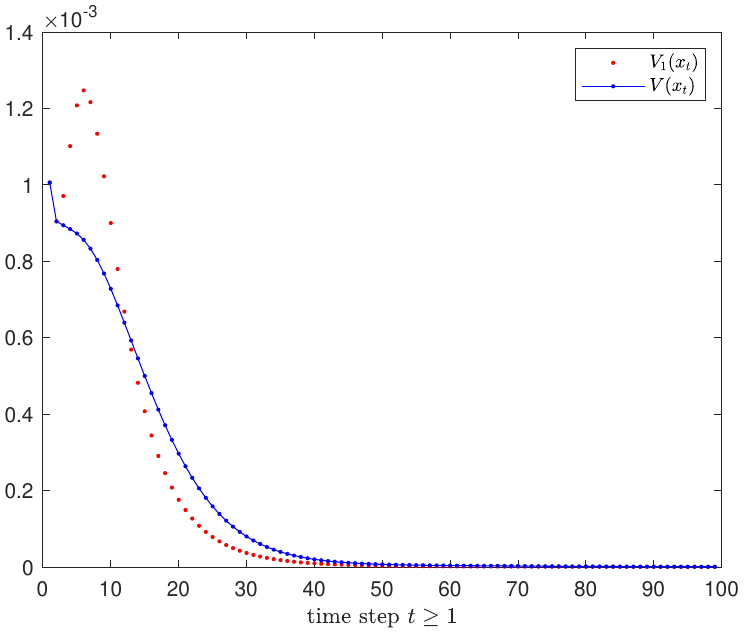} }}%
\caption{Experimental results with $x_0 = (0.9727, 0.9341, 0.0945)^\top$.}%
\label{jet-engine-state-evo}%
\end{figure}

\subsection{Inverted pendulum}
 Finally, we consider the Euler discretization of an inverted pendulum represented in \eqref{n-sys} with $B(x) = \col (0, \frac{T \cos x(1)}{m \ell })$, 
\begin{align*}
   A (x) = \left[
     \begin{smallmatrix} 1 & T \\ \frac{T g \operatorname{sin} x(1)}{\ell x(1)} & 1- \frac{T \mu}{m \ell^2} \end{smallmatrix}\right], \quad x(1) \neq 0     
\end{align*}
where the left-bottom is replaced by $\frac{Tg}{\ell}$ if $x(1) = 0$. 
Next, let $T = 0.1$, $g = 9.8$, $m=1$, $\ell = 1$, $\mu = 0.01$. 
The initial condition is taken as random variables within $x_1(0) \in [-\frac{\pi}{2}, \frac{\pi}{2}]$, $x_2(0) \in [-3, 3]$. 
As shown in Fig.~\ref{invpenstap}, the state variables converge to the origin. 
\begin{figure}[ht!]
\centering
{{\includegraphics[width=6cm]{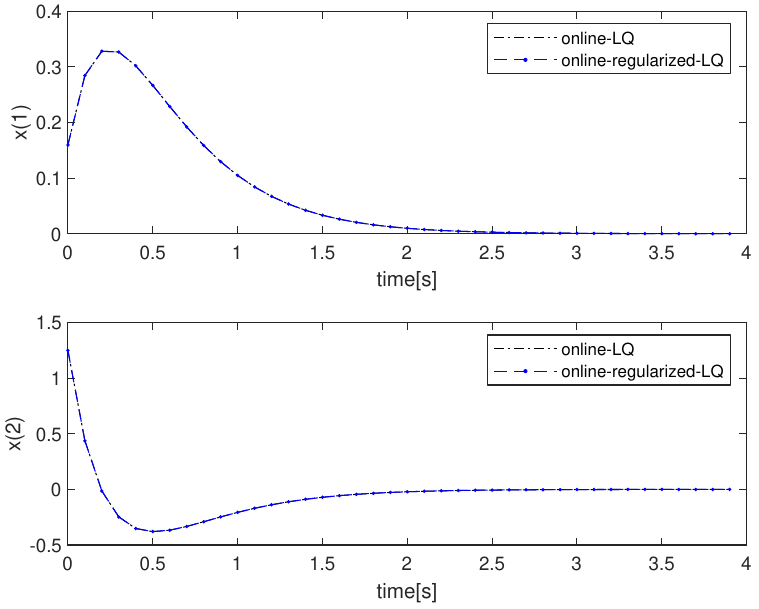} }}%
    \caption{Experimental results with $x_0 = \col(0.1596, 1.2489)$, $\bar \gamma := \frac{\lambda_M(P(x_0))}{\lambda_m (P(x_0))} \gamma(x_0)$. }%
    \label{invpenstap}%
\end{figure}

The simulation demonstrates the effectiveness and flexibility of the time-varying SDP-based online control method. We note that the controller solving \eqref{n-sdp} guarantees closed-loop geometric stability with mild performance degradation compared to the controller solving \eqref{n-lq}.

\section{Conclusions}\label{conclude-sec}

This paper presents a time-varying SDP-based method for stabilizing input-affine nonlinear systems with quadratic performance objectives. The core contributions lie in the formulation of SDPs with explicitly imposed time-varying Lyapunov stability conditions and performance specifications, and in the derivation of compact conditions certifying the recursive feasibility of the proposed SDP sequence. 
The method is generalizable to settings with uncertain model knowledge \cite{ddnijrnc}.  
Future research includes considering nonlinear optimal control problems, where the value function serves as a control Lyapunov function \cite{freeman1995optimal}. 

\section*{Acknowledgements}
The author thanks Prof. C. De Persis and Prof. N. Monshizadeh for valuable discussions and feedback.

\appendix

\subsection{Proof of Proposition \ref{riccati-sdp-lq-property}}
\label{appdx:proof:formulation-relation} 

Since $(A, B)$ is stabilizable, there exists a matrix $K$ that makes $A+BK$ Schur. Then, the matrix 
\begin{align}\label{P-def}
    P := \sum_{i=0}^{\infty} (A+BK)^i ((A+BK)^\top)^i 
\end{align}
is well-defined. Next, let 
\begin{align}\label{L-gamma-def}
    L := KPK^\top, \quad \gamma := \tr(RL) + \tr(QP).
\end{align}
The solution $(K, P, L, \gamma)$ is a feasible solution to \eqref{lmi-lq} since it satisfies the constraints \eqref{lti-stab1}-\eqref{lti-perf2}. 

Next, we specialize the stabilizing $K$ to the one defined as in \eqref{lti-lq}, which is well-defined since a unique $P_r \succ 0$ in \eqref{lti-lq} exists as $(A, B)$ is stabilizable and $(Q^{1/2}, A)$ is observable \cite[Thm. 6.3.2]{chen1995optimal}. 
We denote this choice of $K$ by $K_*$, and note that $K_*$ is the unique minimizer of $\mathcal{H}_2$ norm of the system \eqref{h2phisys} \cite[Sec. 6.4]{chen1995optimal}.  Consistently, we denote by $P_*$, $L_*$, $\gamma_*$, the quantities defined in \eqref{P-def} and \eqref{L-gamma-def} with $K=K_*$. Recall that $(K_*, P_*, L_*, \gamma_*)$ is a feasible solution to \eqref{lmi-lq}. Next, we show that this feasible solution is optimal.
\par 

Let $(K_*, \bar P, \bar L, \bar \gamma)$ be an optimal solution to \eqref{lmi-lq}. 
By \eqref{lti-stab1}, we have 
$(A+BK_*) \bar P  (A+BK_*)^\top - \bar P + I + M = 0$ for some matrix $M\succeq0$, and thus 
\begin{align}\label{p-fea-def}
    \bar P = \sum_{i=0}^{\infty} (A+BK_*)^i (I+M) ((A+BK_*)^\top)^i \succeq P_* .
\end{align}
Then, we have
\begin{align}\label{useful}
\small
{\begin{array}{l}
  \bar \gamma \stackrel{\eqref{lti-perf2}}{\geq} \tr(Q \bar P) + \tr (R\bar L)
  \stackrel{\eqref{lti-perf1}}{\geq} \tr(Q \bar P) + \tr (RK_* \bar P K_*^\top) \stackrel{\eqref{p-fea-def}}{\geq}  \\ 
 \tr(QP_*) + \tr(RK_* P_* K_*^\top) \stackrel{\eqref{L-gamma-def}}{=} \tr(QP_*) + \tr(RL_*) \stackrel{\eqref{L-gamma-def}}{=} \gamma_* . 
\end{array}}
\end{align}
By optimality, we conclude that $\bar \gamma =\gamma^*$.
Then, by \eqref{useful} and noting $\bar P \succeq P_*$, it follows that 
$\bar P = P_*, \bar L = L_*. $ 
This shows that $(K_*, P_*, L_*, \gamma_*)$ is the unique optimal solution to \eqref{lmi-lq} with the corresponding cost $\gamma_*$. 

 Next, we show that $\gamma_* = \tr(P_r)$. By rewriting \eqref{riccati-lti-lq} as 
 $(A+BK_*)^\top P_r (A+BK_{*}) - P_r + Q + K_{*}^\top R K_{*} = 0, $
 we have $P_r = \sum_{i=0}^\infty ((A+BK_{*})^{i})^\top (Q+K_{*}^{\top} R K_{*})) (A+BK_{*})^i$. 
 Next, by the properties of the $\tr(\cdot)$ operator, it follows that 
\[
      {\small
 \begin{aligned}
     \tr(P_r) &= \tr \sum_{i=0}^\infty ((A+BK_{*})^{i})^\top ( Q +K_{*}^{\top} R K_{*})) (A+BK_{*})^{i} 
 \\ &= \tr \sum_{i=0}^\infty ((A+BK_{*})^{i})^\top \left[\begin{smallmatrix} Q^{1/2} & K_{*}^\top R^{1/2} \end{smallmatrix}\right] \cdot I  
     (\star)^\top
 \\ &= \tr \sum_{i=0}^\infty \left[\begin{smallmatrix} Q^{1/2} \\ R^{1/2} K_{*} \end{smallmatrix}\right]   
     (A+BK_{*})^i  \cdot I
     (\star)^\top  \\
&=  \tr \left[\begin{smallmatrix} Q^{1/2} \\ R^{1/2} K_{*} \end{smallmatrix}\right] 
 \underbrace{\sum_{i=0}^\infty   
     (A+BK_{*})^i  \cdot  I
     (\star)^\top  }_{= P_* }  \left[\begin{smallmatrix} Q^{1/2} \\ R^{1/2} K_{*} \end{smallmatrix}\right]^\top \\
&=  \tr  \left[\begin{smallmatrix} Q^{1/2} P_* Q^{1/2} & Q^{1/2} P_* K_*^\top R^{1/2} \\ * & R^{1/2} K_{*} P_* K_*^\top R^{1/2} \end{smallmatrix}\right] \\ 
&= \tr(QP_*) + \tr(RK_* P_* K_*^\top) = \gamma_*
 \end{aligned}}
\]
 which establishes the claim.
\qedp

\subsection{Proof of Proposition \ref{thm-lq-scalar-bound}}\label{pf:prop:bargamma-analytic-scalar}
Note that the scalar system $(a(x), b(x))$ is stabilizable for any $x\in \mathbb{R}$ with $b(x) \neq 0$. By Proposition \ref{riccati-sdp-lq-property}, there exists a unique optimal solution $(k(x), p(x), \ell(x), \gamma(x))$ to \eqref{scalar-lq} satisfying $\gamma(x) = p_r(x)$ where 
\begin{align}\label{xr-sol}
\begin{array}{l}
p_r(x) = 
\frac{a^2(x)+b^2(x)-1+\sqrt{(a^2(x)+b^2(x)-1)^2 + 4b^2(x)}}{2b^2(x)}
\end{array}
\end{align}
is the unique positive definite solution to the Riccati equation \eqref{riccati-lti-lq} with $(A, B)$ replaced by $(a(x), b(x))$. 
By the definition of $p_r$ in \eqref{xr-sol}, we have 
\begin{align}\label{xr-a-b-leq}
{\begin{array}{rl}
    p_r(x) & \stackrel{\eqref{xr-sol}}{=}  \frac{a^2(x) + b^2(x)-1+\sqrt{(a^2(x)+b^2(x)-1)^2 + 4b^2(x)}}{2b^2(x)} \\
      &\leq \frac{a^2(x) + b^2(x)-1 + (a^2(x) + b^2(x)+1)}{2b^2(x)}
    \\ &= \frac{2a^2(x) +2b^2(x)}{2b^2(x)} = 1+\frac{a^2(x)}{b^2(x)}
\end{array}}
\end{align} 
By Assumption \ref{assum-bx-bbarbelow} and \eqref{e:L}, for every $x\in \mathcal{X}$, we have $p_r(x) \stackrel {\eqref{xr-a-b-leq}}{\leq} 1+\frac{a^2(x)}{b^2(x)} \leq
1+\frac{L^2}{\underline b^2}.$ 
Recalling that $\gamma(x)=p_r(x)$, the upper bound in \eqref{gamma-scalar-bound} is established. 
Next, bearing in mind $k(x)$ is equal to the control gain in \eqref{lti-lq} with $A, B, P_r$ replaced by $a(x), b(x), p_r(x)$, we can rewrite the Riccati equation \eqref{riccati-lti-lq} with coefficients $(a(x), b(x))$ as $(a(x)+b(x)k(x))^2 p_r(x)-p_r(x) + 1+k^2(x) = 0.$ 
By computing $p_r(x)$ from the equation, we observe that
$ p_r(x) = \frac{1+k^2(x)}{1-(a(x) + b(x)k(x))^2}.$ 
Moreover, by specializing \eqref{P-def} to the scalar case, we have  
\begin{align}\label{p-lq-def}
    p(x)= \frac{1}{1-(a(x)+b(x)k(x))^2}
\end{align}
and thus
\begin{align}\label{x-p-rel-scalar}
       p_r(x) = (1+k^2(x))p(x). 
\end{align}
Therefore, we find that
\begin{align}\label{p-lq-bound-1}
    \gamma(x) = (1+k^2(x))p(x) \geq p(x) \geq 1,
\end{align}
where the second inequality follows from \eqref{p-lq-def} and the fact that $a(x)+b(x)k(x)$ is Schur. 
This completes the proof. 
\qedp

\subsection{Proof of Proposition \ref{redef-sdp-property}}\label{pf:prop:regu-sdp}
First, we show that the feasibility set of \eqref{n-sdp} is bounded. 
By \eqref{n-sdp-stab1}, we have $P(x_t) \succeq I$ and thus, from \eqref{n-sdp-perf2}, it follows that $\gamma(x_t)  \geq \operatorname{tr}(P(x_t) ) \geq \operatorname{tr}(I_n) \geq n$. 
Next, recall that $\specnorm{M} \leq \operatorname{tr}(M)$ for any positive semidefinite matrix $M$. Then, $\specnorm{P(x_t)} \leq \operatorname{tr}(P(x_t)) \stackrel{\eqref{n-sdp-perf2}}{\leq} \gamma(x_t) \stackrel{\eqref{n-sdp-stab4}}{\leq} \bar \gamma $.
Analogously, it gives $\specnorm{L(x_t)} \leq \operatorname{tr}(L(x_t)) \stackrel{\eqref{n-sdp-perf2}}{\leq}
\gamma(x_t) - \operatorname{tr}(P(x_t))  \stackrel{\eqref{n-sdp-stab4}}{\leq} \bar \gamma-n$. 
Next, we show that $\specnorm{Y(x_t)}^2 \leq \bar \gamma (\bar \gamma-n)$. By \eqref{n-sdp-perf1}, we have 
\[
Y(x_t) P^{-1}(x_t) Y^\top(x_t) \preceq L(x_t)
\]
and thus $\operatorname{tr} (Y(x_t) P^{-1}(x_t) Y^\top(x_t)) \leq \operatorname{tr}(L(x_t)) \leq \bar \gamma - n$. Moreover, since $P^{-1}(x_t) \succeq \lambda_m (P^{-1}(x_t)) I_n$, this gives 
\[{
\operatorname{tr}(Y(x_t) P^{-1}(x_t) Y^\top(x_t)) \geq \lambda_m (P^{-1}(x_t)) \operatorname{tr}(Y^\top(x_t) Y(x_t)).}
\]  
Then 
\[
\begin{array}{rl}
\operatorname{tr}(Y(x_t) P^{-1}(x_t) Y^\top(x_t)) \geq \lambda_m (P^{-1}(x_t)) \operatorname{tr}(Y^\top (x_t) Y(x_t)) 
\\
= \lambda_M^{-1} (P(x_t)) \operatorname{tr}(Y^\top(x_t) Y(x_t)) \geq \bar \gamma^{-1} \operatorname{tr}(Y^\top(x_t) Y(x_t)).
\end{array}
\]
Hence, $\operatorname{tr}(Y^\top(x_t) Y(x_t)) \leq \bar \gamma (\bar \gamma - n)$, and thus  $\specnorm{Y(x_t)}^2 \leq \operatorname{tr}(Y^\top (x_t)Y(x_t)) \leq \bar \gamma (\bar \gamma - n)$. 
This completes the proof of the fact that the feasibility set is bounded. Therefore, bearing in mind that the feasibility set is nonempty by assumption, we find that there exists an optimal solution $(Y(x_t), P(x_t), L(x_t), \gamma(x_t))$ to the convex program \eqref{n-sdp}. 
\par
Next, we show that any optimal solution $(Y(x_t), P(x_t), L(x_t), \gamma(x_t))$ to \eqref{n-sdp} has the properties claimed in \eqref{opt-sol-bound-P} and \eqref{decre-lyap}. 
By recalling the definition of $A_{cl}(x_t)$ in \eqref{n-close-sys} and using Schur complement, \eqref{n-sdp-stab1} implies $P(x_t) \succeq I$ and
\begin{align}\label{pf-useful-kkkkkk}
  A_{cl}(x_t) P(x_t)A_{cl}^\top(x_t) -  P(x_t) + I \preceq 0. 
\end{align}
By applying a congruent transformation, the above matrix inequality can be equivalently expressed as 
\begin{align}\label{guodu}
    P^{-1}(x_t) A_{cl}(x_t)\cdot P(x_t)  (\star)^\top - P^{-1}(x_t) + P^{-2}(x_t) \preceq 0 .
\end{align}
and thus $P^{-1}(x_t) A_{cl}(x_t)\cdot P(x_t)  (\star)^\top - P^{-1}(x_t) \prec 0$ since $P^{-2}(x_t) \succ 0$. Then, we have 
\[
\left[ \begin{smallmatrix}
    -P^{-1}(x_t) & A_{cl}^\top(x_t) P^{-1}(x_t) \\ P^{-1}(x_t) A_{cl}(x_t) 
     & - P^{-1}(x_t)
\end{smallmatrix} \right] \prec 0
\]
By Schur complement, the inequality \eqref{decre-lyap} holds. 
Finally, we show that under the nonsingularity condition on $P(x_t) - I$, \eqref{decre-geo-lyap} holds. By \eqref{guodu}, it holds that
$ P^{-1}(x_t) A_{cl}(x_t)\cdot P(x_t)  (\star)^\top \preceq P^{-1}(x_t) - P^{-2}(x_t) \preceq (1 - \lambda_m(P^{-1}(x_t))) P^{-1}(x_t)$. 
It follows that 
\begin{align}\label{pf-use-kkk-ggg}
   \left[ \begin{smallmatrix} -P^{-1}(x_t) & A_{cl}^\top(x_t) P^{-1}(x_t) \\ * & -(1-\lambda_m(P^{-1}(x_t)))P^{-1}(x_t)
\end{smallmatrix} \right] \preceq 0
\end{align}
Next, since $P(x_t) - I$ is nonsingular and $P(x_t)\succeq I$, then $P(x_t) \succ I$, and thus
\begin{align}\label{pf-ssjsjshs}
    \lambda_m(P^{-1}(x_t)) = \lambda_M^{-1} (P(x_t)) < 1.
\end{align}
Hence, $(1-\lambda_m(P^{-1}(x_t)))P^{-1}(x_t)$ is invertible. Then, by Schur complement, \eqref{pf-use-kkk-ggg} implies  
$-P^{-1}(x_t) + A_{cl}^\top(x_t) P^{-1}(x_t) (1-\lambda_m(P^{-1}(x_t)))^{-1}\cdot P(x_t)  (\star)^\top \preceq 0$. 
Multiplying the inequality by $(1-\lambda_m(P^{-1}(x_t))$, we have
$(1-\lambda_m(P^{-1}(x_t))) P^{-1}(x_t) - A_{cl}^\top(x_t) P^{-1}(x_t)A_{cl}(x_t) \succeq 0$, which gives \eqref{decre-geo-lyap}. This completes the proof.  
\qedp

\subsection{A recursive feasibility condition for input affine systems with state dependent input vector fields}\label{appdx:rf:Bxcase}

Consider the system \eqref{n-sys}. 
Let $B(x)$ be partitioned as $B(x) = \bar B + \hat{B}(x)$ with $\hat{B}: \R^n \rightarrow \R^{n \times m}$ a matrix-valued function. 
Proposition \ref{prop:offline:lmi:rf} extends to this case as follows.

\begin{cor}
Let $\Delta \in \R^{n \times r_1}$, $\Delta_B \in \R^{n \times r_2}$ be known matrices and $\mathcal{X} \subseteq \R^n$ a compact set including the origin such that 
\begin{align}\label{hatA-strength-condition-hatA-hatB}
\hat{A}(x) \hat{A}(x)^\top \preceq \Delta \Delta^\top , 
\hat{B}(x) \hat{B}(x)^\top \preceq \Delta_B \Delta_B^\top, 
\forall x \in \mathcal{X}. 
\end{align} 
Consider the following LMI in the decision variables $\epsilon_1 > 0$, $\epsilon_2 > 0$, $P \in \mathbb{S}^{n\times n}$, and $Y \in \R^{m \times n}$: 
\begin{align}\label{converge-lmi-state-dependent-B}
\begin{bmatrix}
    P - I - \epsilon_1 \Delta \Delta^\top - \epsilon_2 \Delta_B \Delta_B^\top & \bar{A} P + \bar B Y & 0 & 0 \\[.2em] (\bar{A} P + \bar B Y)^\top & P & Y^\top & P \\[.2em] 
    0 & Y & \epsilon_2 I & 0 \\ 0 & P & 0 & \epsilon_1 I \end{bmatrix} \succeq 0
\end{align}
Suppose \eqref{converge-lmi-state-dependent-B} is feasible and let $(\epsilon_1, \epsilon_2, P, Y)$ be any feasible solution. Then, for any $x_0 \in \mathcal{R}_\alpha := \{x\in \R^n: x^\top P^{-1} x \le \alpha, \alpha > 0\}$ contained in $\mathcal{X}$, the time-varying SDP \eqref{n-sdp} with 
\[
\bar \gamma \ge \operatorname{tr}(P) + 
\operatorname{tr}(Y P^{-1} Y^\top)
\]
is feasible for all $t \ge 0$. \qeds
\end{cor}

\addtolength{\textheight}{-10cm}   


\bibliographystyle{IEEEtran}
\bibliography{mybib}

\end{document}